\documentclass[12pt,reqno]{article}

\usepackage[usenames]{color}
\usepackage{amssymb}
\usepackage{graphicx}
\usepackage{amscd}
\usepackage{amsmath}
\usepackage{amsthm}

\usepackage{rotating}
\usepackage{float}

\usepackage[colorlinks=true,
linkcolor=webgreen,
filecolor=webbrown,
citecolor=webgreen]{hyperref}

\definecolor{webgreen}{rgb}{0,.5,0}
\definecolor{webbrown}{rgb}{.6,0,0}

\begin{document}

\renewcommand{\baselinestretch}{1.0}\normalsize

\theoremstyle{plain}
\newtheorem{theorem}{Theorem}
\newtheorem{lemma}[theorem]{Lemma}
\newtheorem{corollary}[theorem]{Corollary}
\theoremstyle{definition}
\newtheorem{definition}[theorem]{Definition}
\newtheorem{example}[theorem]{Example}
\newtheorem{remark}[theorem]{Remark}

\begin{center}
\vskip 1cm
{\LARGE\bf
Balancing polynomials, Fibonacci numbers and some new series for $\pi$} \\

\vskip 1cm \large
Robert Frontczak\footnote{Corresponding author.
Statements and conclusions made in this paper by R. F. are entirely those of the author. They do not necessarily reflect the views of LBBW.}

Landesbank Baden-W{\"u}rttemberg \\ Am Hauptbahnhof 2, 70173 Stuttgart \\ Germany \\
\href{mailto:robert.frontczak@lbbw.de}{\tt robert.frontczak@lbbw.de}

\vskip .2in
Kalika Prasad \\
Central University of Jharkhand, \\ Ranchi, 835205, India \\
\href{mailto:klkaprsd@gmail.com}{\tt klkaprsd@gmail.com} \\
\end{center}

\vskip .2 in

\begin{abstract}
We evaluate some types of infinite series with balancing and Lucas-balancing polynomials in closed form. 
These evaluations will lead to some new curious series for $\pi$ involving Fibonacci and Lucas numbers.
Our findings complement those of Castellanos from 1986 and 1989. 
\end{abstract}



\section{Motivation and Preliminaries}

Castellanos \cite{Cas1,Cas2} has found, among other things, the following curious series for $\pi$:
\begin{equation}\label{cas}
\frac{\pi}{4} = \sum_{n=0}^\infty \frac{(-1)^n}{2n+1}\frac{F_{(2n+1)(2m+1)}}{5^n}
\Big (\frac{2}{F_{2m+1} + \sqrt{F_{2m+1}^2+4/5}}\Big )^{2n+1}, \quad m\geq 0.
\end{equation}
When $m=0$ the expression simplifies to (\cite[Eq. (46)]{Cas1} and \cite[Eq. (3.4)]{Cas2})
\begin{equation*}
\frac{\pi}{4} = \sqrt{5} \sum_{n=0}^\infty \frac{(-1)^n}{2n+1}F_{2n+1} \Big (\frac{2}{3 + \sqrt{5}}\Big )^{2n+1},
\end{equation*}
which has the equivalent form
\begin{equation}
\frac{\pi}{4} = \sqrt{5} \sum_{n=0}^\infty \frac{(-1)^n}{2n+1}\frac{F_{2n+1}}{\alpha^{4n+2}}.
\end{equation}
Another series involving the squares of $F_{2n+1}$ is (\cite[Eqs. (3.2) and (3.3)]{Cas2})
\begin{equation}\label{cas2}
\frac{\pi}{20} = \sum_{n=0}^\infty \frac{(-1)^n}{2n+1}\frac{F_{(2n+1)(2m+1)}^2}{(t+\sqrt{t^2+1})^{2n+1}}, \quad m\geq 0,
\end{equation}
with
\begin{equation*}
t = \frac{5F_{2m+1}^2}{4} \Big (1+\sqrt{1+\Big (\frac{16}{25}\Big ( \frac{5}{2}F_{2m+1}^2-1\Big)/F_{2m+1}^4 \Big )},
\end{equation*}
which, for $m=0$, simplifies to
\begin{equation}\label{cas2_1}
\frac{\pi}{20} = \sum_{n=0}^\infty \frac{(-1)^n}{2n+1} \frac{F_{2n+1}^2}{(3+\sqrt{10})^{2n+1}}.
\end{equation}
In the identities above, $F_n$ (respectively $L_n$) are the famous Fibonacci (Lucas) numbers, 
defined for $n\geq 0$ by the recursion $F_{n+2} = F_{n+1} + F_n$ ($L_{n+2} = L_{n+1} + L_n$) 
with initial conditions $F_0 = 0, F_1 = 1$ ($L_0=2$ and $L_1=1$), and $\alpha=\frac{1+\sqrt{5}}{2}$ is the golden ratio. 
Castellanos' series pose the question, whether more series of these types exist? 
Here, we give a positive answer to this question.
Working with balancing and Lucas-balancing polynomials we derive more series for $\pi$ 
involving Fibonacci and Lucas numbers and exhibiting such a structure. 

Recall that, for any integer $n\geq0$ and $x\in\mathbb{C}$, balancing polynomials $\big(B_n(x)\big)_{n\geq0}$ and
Lucas-balancing polynomials $\big(C_n(x)\big)_{n\geq0}$ are defined by the second-order homogeneous linear recurrence \cite{Fro1}
\begin{equation}\label{Bal_def}
u_n(x) = 6x u_{n-1}(x) - u_{n-2}(x),
\end{equation}
but with different initial terms. Balancing polynomials start with $B_0(x)=0$ and $B_1(x)=1$, while for Lucas-balancing polynomials
we set $C_0(x)=1$, $C_1(x)=3x$. These polynomials have been introduced as a natural extension of the popular balancing and Lucas-balancing numbers $B_n$ and $C_n$, respectively, and must be seen as a special member of the Horadam sequence.
Obviously, $B_n=B_n(1)$ and $C_n=C_n(1)$. The first few polynomials are
\begin{gather*}
B_0(x)=0,\quad\,\, B_1(x)=1,\quad\,\, B_2(x)=6x,\quad\,\, B_3(x)=36x^2-1,\\ B_4(x)=216x^3-12x, \,\,\quad B_5(x)=1296x^4-108x^2+1,
\end{gather*}
and
\begin{gather*}
C_0(x)=1,\,\,\quad C_1(x)=3x,\,\,\quad C_2(x)=18x^2-1,\,\,\quad C_3(x)=108x^3-9x,\\
C_4(x)=648x^4-72x^2+1,\,\,\quad C_5(x)=3888x^5-540x^3+15x.
\end{gather*}

The closed forms, known as Binet's formulas, for balancing and Lucas-balancing polynomials are given by
\begin{equation}\label{Binet}
B_n(x) = \frac{\lambda_1^n (x) - \lambda_2^{n} (x)}{\lambda_1(x)-\lambda_2 (x)},\qquad
C_n(x) = \frac{\lambda_1^n (x) + \lambda_2^{n} (x)}{2},
\end{equation}
where $\lambda_1(x)=3x + \sqrt{9x^2-1}$ and $\lambda_2(x)=3x - \sqrt{9x^2-1}$. We note the following properties:
$\lambda_1(x)\cdot\lambda_2(x)=1,\lambda_1(x)-\lambda_2(x)=2\sqrt{9x^2-1}, \lambda_1(x)+\lambda_2(x)=6x$
and $0<\lambda_2(x)<\lambda_1(x)$ for $x>1/3$. The polynomials possess a simple connection to Chebyshev polynomials \cite{Fro1} via 
\begin{equation*}
B_n(x) = U_{n-1}(3x) \qquad \mbox{and} \qquad C_n(x) = T_n(3x),
\end{equation*}
where $U_n(x)$ and $T_n(x)$ are Chebyshev polynomials of the first and second kind, respectively.
Some other interesting properties have been discovered in the articles \cite{Fro2,Fro3,FroGoy1,FroGoy2,Kim1}.
Balancing and Lucas-balancing polynomials are also linked in various ways to Fibonacci and Lucas numbers.
Two such connections, which will be exploited in the text, are
\begin{equation} \label{bf1}
B_n\Big (\frac{L_{2m}}{6} \Big ) = \frac{F_{2mn}}{F_{2m}} \,\, , \qquad C_n\Big (\frac{L_{2m}}{6} \Big ) = \frac{L_{2mn}}{2},
\end{equation}
and
\begin{equation} \label{bf2}
B_{2n+1}\Big (\frac{\sqrt{5}}{6} F_{2m+1} \Big ) = \frac{L_{(2m+1)(2n+1)}}{L_{2m+1}} \,\, , \qquad
C_{2n+1}\Big (\frac{\sqrt{5}}{6}F_{2m+1} \Big ) = \frac{\sqrt{5}}{2}F_{(2m+1)(2n+1)}.
\end{equation}
These relations can be deduced from the corresponding relations between Fibonacci numbers and Chebyshev polynomials
(see \cite{Cas2} and \cite{Fro1}).

In the present article, we study special infinite series involving $B_n(x)$ and $C_n(x)$, respectively.
We evaluate these series in closed form. Based on these results, new infinite series for $\pi$ with Fibonacci and
Lucas numbers will be stated as an immediate consequence. These new and curious series must be seen as complements
of Castellanos' results from 1986 and 1989.

\section{Complementing the work of Castellanos (Part 1)}

The next theorem is our starting point.

\begin{theorem} \label{thmC}
For each real $x$ with $x>1/3$ and $z\in\mathbb{C}$ with $|z|>|\lambda_1(x)|$, we have 
\begin{equation}\label{main1C}
\sum_{n=0}^\infty \frac{(-1)^{n}}{2n+1}\frac{B_{2n+1}(x)}{z^{2n+1}} = \frac{1}{2\sqrt{9x^2-1}}\arctan\Big (\frac{2\sqrt{9x^2-1}z}{z^2+1}\Big )
\end{equation}
and
\begin{equation}\label{main2C}
\sum_{n=0}^\infty \frac{(-1)^{n}}{2n+1}\frac{C_{2n+1}(x)}{z^{2n+1}} = \frac{1}{2}\arctan\Big (\frac{6xz}{z^2-1}\Big ).
\end{equation}
\end{theorem}
\begin{proof}
Recall the Taylor series for the arctangent function
\begin{equation}\label{taylor}
\arctan(z) = \sum_{n=0}^\infty \frac{(-1)^{n}}{2n+1} z^{2n+1}, \quad |z|\leq 1.
\end{equation}
Hence, for all $x$ and $z$ with $|z|>|\lambda_1(x)|$,
\begin{eqnarray*}
\sum_{n=0}^\infty \frac{(-1)^{n}}{2n+1}\frac{B_{2n+1}(x)}{z^{2n+1}} & = &
\frac{1}{\lambda_1(x)-\lambda_2(x)} \Big ( \arctan\Big (\frac{\lambda_1(x)}{z}\Big ) - \arctan\Big (\frac{\lambda_2(x)}{z}\Big )\Big ) \\
& = & \frac{1}{2\sqrt{9x^2-1}}\arctan\Big (\frac{2\sqrt{9x^2-1}z}{z^2+1}\Big ),
\end{eqnarray*}
where in the last step we applied the identity
\begin{equation}\label{arc_id}
\arctan(a) \pm \arctan(b) = \arctan\Big(\frac{a \pm b}{1 \mp ab}\Big).
\end{equation}
This proves the first identity. The proof of the second identity is very similar and omitted.
\end{proof}

\begin{corollary}
The following series representation for $\pi$ holds
\begin{equation}\label{Luc_bal_pi}
\frac{\pi}{8} = \sum_{n=0}^\infty \frac{(-1)^n}{2n+1} \frac{C_{2n+1}}{(3+\sqrt{10})^{2n+1}}.
\end{equation}
\end{corollary}
\begin{remark}
Observe the similarity of equation \eqref{Luc_bal_pi} to Castellanos' series \eqref{cas2_1}.
To be particular, the two series provide an interesting example for series whose members’ denominators 
and the sum are the same but the numerators are different. Namely,
\begin{equation*}
\sum_{n=0}^\infty \frac{(-1)^n}{2n+1} \frac{F_{2n+1}^2}{(3+\sqrt{10})^{2n+1}} = \frac{\pi}{20} = 
\sum_{n=0}^\infty \frac{(-1)^n}{2n+1} \frac{\frac{2}{5} C_{2n+1}}{(3+\sqrt{10})^{2n+1}}.
\end{equation*}
Such series seem to be rare. Another example given by Mez\H{o} \cite{Mezo} (and involving Lucas numbers) 
are the series
\begin{equation*}
\sum_{n=1}^\infty \frac{L_n}{n2^n} = 2\cdot \ln(2) = \sum_{n=1}^\infty \frac{2}{n2^n}.
\end{equation*}
\end{remark}

\subsection{Series with even Fibonacci (Lucas) coefficients}

Setting $x=L_{2m}/6, m\geq 0$, we first note that
\begin{equation*}
\lambda_1\Big (\frac{L_{2m}}{6}\Big ) = \frac{L_{2m}+\sqrt{L_{2m}^2-4}}{2} = \alpha^{2m},
\end{equation*}
with $\alpha=(1+\sqrt{5})/2$ and where we have used that $L_n^2=5F_n^2+(-1)^n4$.
We can now use \eqref{bf1} to obtain the relations, valid for all $|z|>\alpha^{2m}$,
\begin{equation}\label{main1FibC}
\sum_{n=0}^\infty \frac{(-1)^n}{2n+1}\frac{F_{2m(2n+1)}}{z^{2n+1}} = \frac{1}{\sqrt{5}}\arctan\Big (\sqrt{5} \frac{F_{2m}z}{z^2+1}\Big ),
\end{equation}
and
\begin{equation}\label{main1LucC}
\sum_{n=0}^\infty \frac{(-1)^n}{2n+1}\frac{L_{2m(2n+1)}}{z^{2n+1}} = \arctan\Big (\frac{L_{2m}z}{z^2-1}\Big ).
\end{equation}
Especially, for $z=L_{2m}$,
\begin{equation}
\sum_{n=0}^\infty \frac{(-1)^n}{2n+1} \frac{F_{2m(2n+1)}}{L_{2m}^{2n+1}} = 
\frac{1}{\sqrt{5}}\arctan\Big (\sqrt{5}\frac{F_{4m}}{L_{2m}^2+1}\Big ),
\end{equation}
and
\begin{equation}
\sum_{n=0}^\infty \frac{(-1)^n}{2n+1}\frac{L_{2m(2n+1)}}{L_{2m}^{2n+1}} = \arctan\Big (\frac{L_{2m}^2}{L_{2m}^2-1}\Big ).
\end{equation}
Next, from $\alpha^n=\alpha F_n + F_{n-1},n\geq 1,$ we see that
\begin{equation*}
\alpha^{2m} < 2F_{2m} + F_{2m-1} = F_{2m+2}.
\end{equation*}
Hence, inserting $z=F_{2m+2}$, we also get the relations
\begin{equation}
\sum_{n=0}^\infty \frac{(-1)^n}{2n+1} \frac{F_{2m(2n+1)}}{F_{2m+2}^{2n+1}} =
\frac{1}{\sqrt{5}}\arctan\Big (\sqrt{5}\frac{F_{2m}F_{2m+2}}{F_{2m+1}F_{2m+3}}\Big ),
\end{equation}
and
\begin{equation}
\sum_{n=0}^\infty \frac{(-1)^n}{2n+1}\frac{L_{2m(2n+1)}}{F_{2m+2}^{2n+1}} =
\arctan\Big (\frac{L_{2m}F_{2m+2}}{F_{2m}F_{2m+4}}\Big ),
\end{equation}
where we have used the Catalan identities
\begin{equation*}
F_{2m+2}^2+1 = F_{2m+3}F_{2m+1} \quad\mbox{and}\quad F_{2m+2}^2-1 = F_{2m+4}F_{2m}.
\end{equation*}

To derive Castellanos-like expressions for $\pi$ with Fibonacci (Lucas) coefficients,
we use \eqref{main1FibC} and \eqref{main1LucC}, and relate them to special arctan arguments. 
We start with
\begin{equation*}
\frac{\sqrt{5}F_{2m}z}{z^2+1}=1 \qquad \mbox{and} \qquad \frac{L_{2m}z}{z^2-1}=1.
\end{equation*}
Solving $z^2-\sqrt{5}F_{2m}z+1=0,m\geq 1,$ gives
\begin{equation*}
z_{1/2}=\frac{\sqrt{5}F_{2m}\pm \sqrt{5F_{2m}^2-4}}{2}.
\end{equation*}
But it can be checked easily that for all $m\geq 1$
\begin{equation*}
0 < z_2 = \frac{\sqrt{5}F_{2m} - \sqrt{5F_{2m}^2-4}}{2} < z_1 = \frac{\sqrt{5}F_{2m} + \sqrt{5F_{2m}^2-4}}{2} < \alpha^{2m}
\end{equation*}
with $\lim_{m\rightarrow\infty} z_2 = 0$ and $z_1\sim \alpha^{2m}$ as $m\rightarrow\infty$. 
Similarly, solving the equation $z^2-L_{2m}z-1=0,m\geq 0$ gives
\begin{equation*}
z_{1/2}=\frac{L_{2m}\pm \sqrt{L_{2m}^2+4}}{2}
\end{equation*}
and 
\begin{equation*}
z_2 = \frac{L_{2m} - \sqrt{L_{2m}^2+4}}{2} < 0 < \alpha^{2m} < z_1 = \frac{L_{2m} + \sqrt{L_{2m}^2+4}}{2}.
\end{equation*}
Note also that $\lim_{m\rightarrow\infty} z_2 = 0$ while $z_1\sim \alpha^{2m}$ as $m\rightarrow\infty$. 
Hence, inserting $z_1$ in \eqref{main1LucC} and simplifying proves the next result.

\begin{theorem} \label{thm2C}
For each integer $m\geq 0$ the following expression for $\pi$ is valid
\begin{equation}\label{main2LucC}
\frac{\pi}{4} = \sum_{n=0}^\infty \frac{(-1)^n}{2n+1} L_{2m(2n+1)} \Big (\frac{2}{L_{2m} + \sqrt{L_{2m}^2+4}}\Big )^{2n+1}.
\end{equation}
\end{theorem}

Equation \eqref{main2LucC} is the first Lucas number counterpart of \eqref{cas}. 
It is worth to note, that the case $m=0$ yields
\begin{equation*}
\frac{\pi}{8} = \sum_{n=0}^\infty \frac{(-1)^n}{2n+1} \Big (\frac{1}{1 + \sqrt{2}}\Big )^{2n+1},
\end{equation*}
which can be obtained directly from \eqref{taylor} with $z=1/(1+\sqrt{2})$ and using the arctan identity
\begin{equation*}
2 \arctan(x) = \arctan\Big ( \frac{2x}{1-x^2}\Big ), \qquad x^2<1.
\end{equation*}
When $m=1$, then
\begin{equation*}
\frac{\pi}{4} = \sum_{n=0}^\infty \frac{(-1)^n}{2n+1}L_{4n+2} \Big (\frac{2}{3 + \sqrt{13}}\Big )^{2n+1}.
\end{equation*}

In a similar manner, working with other arguments of the arctan function, 
we can derive the following presumably new series for $\pi$ involving even Lucas numbers.

\begin{theorem} \label{thm3C}
For each integer $m\geq 0$ the following expressions for $\pi$ are valid
\begin{equation}\label{main3LucC}
\frac{\pi}{6} = \sum_{n=0}^\infty \frac{(-1)^n}{2n+1} L_{2m(2n+1)} \Big (\frac{2}{\sqrt{3}L_{2m} + \sqrt{3 L_{2m}^2+4}}\Big )^{2n+1},
\end{equation}
\begin{equation}\label{main4LucC}
\frac{\pi}{12} = \sum_{n=0}^\infty \frac{(-1)^n}{2n+1} L_{2m(2n+1)} 
\Big (\frac{2(2-\sqrt{3})}{L_{2m} + \sqrt{L_{2m}^2+4(2-\sqrt{3})^2}}\Big )^{2n+1},
\end{equation}
and
\begin{equation}\label{main5LucC}
\frac{\pi}{5} = \sum_{n=0}^\infty \frac{(-1)^n}{2n+1} L_{2m(2n+1)} 
\Big (\frac{2\sqrt{5-2\sqrt{5}}}{L_{2m} + \sqrt{L_{2m}^2+4(5-2\sqrt{5})}}\Big )^{2n+1}.
\end{equation}
\end{theorem}

\subsection{Series with odd Fibonacci (Lucas) coefficients}

Setting $x=\sqrt{5}F_{2m+1}/6, m\geq 0$, we observe that
\begin{equation*}
\lambda_1\Big (\frac{\sqrt{5}F_{2m+1}}{6}\Big ) = \frac{\sqrt{5}F_{2m+1}+L_{2m+1}}{2} = \alpha^{2m+1},
\end{equation*}
where again we have used $L_n^2=5F_n^2+(-1)^n4$.
Proceeding as before, inserting the value of $x$ into \eqref{main1C} and \eqref{main2C} we obtain the following
expressions via \eqref{bf2}, valid for all $|z|>\alpha^{2m+1}$,
\begin{equation}\label{mainFib_odd}
\sum_{n=0}^\infty \frac{(-1)^n}{2n+1}\frac{F_{(2m+1)(2n+1)}}{z^{2n+1}} =
\frac{1}{\sqrt{5}}\arctan\Big (\sqrt{5}F_{2m+1} \frac{z}{z^2-1}\Big ),
\end{equation}
and
\begin{equation}\label{mainLuc_odd}
\sum_{n=0}^\infty \frac{(-1)^n}{2n+1}\frac{L_{(2m+1)(2n+1)}}{z^{2n+1}} = \arctan\Big (L_{2m+1}\frac{z}{z^2+1}\Big ).
\end{equation}

Especially, for $z=L_{2m+1}$,
\begin{equation}
\sum_{n=0}^\infty \frac{(-1)^n}{2n+1} \frac{F_{(2m+1)(2n+1)}}{L_{2m+1}^{2n+1}} = 
\frac{1}{\sqrt{5}}\arctan\Big (\sqrt{5}\frac{F_{4m+2}}{L_{2m+1}^2+1}\Big ),
\end{equation}
and
\begin{equation}
\sum_{n=0}^\infty \frac{(-1)^n}{2n+1}\frac{L_{(2m+1)(2n+1)}}{L_{2m+1}^{2n+1}} = \arctan\Big (\frac{L_{2m+1}^2}{L_{2m+1}^2-1}\Big ).
\end{equation}

Once more, deriving Castellanos-like expressions for $\pi$ is fairly simple.
We start with
\begin{equation*}
\frac{\sqrt{5}F_{2m+1}z}{z^2-1}=1 \qquad \mbox{and} \qquad \frac{L_{2m+1}z}{z^2+1}=1.
\end{equation*}
Solving $z^2-\sqrt{5}F_{2m+1}z-1=0,m\geq 1,$ gives
\begin{equation*}
z_{1/2}=\frac{\sqrt{5}F_{2m+1}\pm \sqrt{5F_{2m+1}^2+4}}{2}.
\end{equation*}
But it can be checked easily that $z_{2}<\alpha^{2m+1}<z_{1}$ for all $m\geq 0$.
Inserting the value $z_1$ in \eqref{mainFib_odd} gives Castellanos' series \eqref{cas}.
Similarly, solving the equation $z^2-L_{2m+1}z+1=0,m\geq 0,$ gives
\begin{equation*}
z_{1/2}=\frac{L_{2m+1}\pm \sqrt{L_{2m+1}^2-4}}{2}
\end{equation*}
We note that, for $m=0$ the zeros $z_{1/2}$ are complex numbers with $|z_1|=|z_2|=1<\alpha$.
For $m\geq 1$, the zeros are real and $z_2 < z_1 < \alpha^{2m+1}$. Hence, no such series for $\pi$ will exist.

In the same manner other arctan arguments can be analyzed. 
Interestingly, the careful analysis shows that no series involving odd Lucas numbers will come to appearance. 
The next theorem contains additional odd-indexed Fibonacci series for $\pi$ that we found.

\begin{theorem} \label{thm4C}
For each integer $m\geq 0$ the following expressions for $\pi$ are valid
\begin{equation}\label{main1Fib}
\frac{\pi}{12} = \sqrt{5}\sum_{n=0}^\infty \frac{(-1)^n}{2n+1} F_{(2m+1)(2n+1)} \Big (\frac{2(2-\sqrt{3})}{\sqrt{5}F_{2m+1} + \sqrt{5F_{2m+1}^2+4(2-\sqrt{3})^2}}\Big )^{2n+1},
\end{equation}
\begin{equation}\label{main2Fib}
\frac{\pi}{6} = \sqrt{5}\sum_{n=0}^\infty \frac{(-1)^n}{2n+1} F_{(2m+1)(2n+1)} \Big (\frac{2}{\sqrt{15}F_{2m+1} + \sqrt{15F_{2m+1}^2+4}}\Big )^{2n+1},
\end{equation}
and
\begin{equation}\label{main3Fib}
\frac{\pi}{5} = \sqrt{5}\sum_{n=0}^\infty \frac{(-1)^n}{2n+1} F_{(2m+1)(2n+1)} \Big (\frac{2(\sqrt{5-2\sqrt{5}})}{\sqrt{5}F_{2m+1} + \sqrt{5F_{2m+1}^2+4(5-2\sqrt{5})}}\Big )^{2n+1}.
\end{equation}
\end{theorem}

The special case of \eqref{main1Fib} when $m=0$ appears in one of Castellanos' papers \cite[Eq. (48)]{Cas1}. 
The generalization and the other two series seem to be new. 

\section{Complementing the work of Castellanos (Part 2)}

Replacing $z$ in \eqref{main1C} by $z/\lambda_1(x)$ and $z/\lambda_2(x)$, respectively, and combining the terms according to the Binet form 
gives
\begin{equation*}
\sum_{n=0}^\infty \frac{(-1)^{n}}{2n+1}\frac{B_{2n+1}^2(x)}{z^{2n+1}} = \frac{1}{4(9x^2-1)} \Big ( \arctan\Big (
\frac{2\sqrt{9x^2-1} \lambda_1(x) z}{z^2+\lambda_1^2(x)} \Big ) - \arctan\Big ( \frac{2\sqrt{9x^2-1} \lambda_2(x) z}{z^2+\lambda_2^2(x)}\Big )\Big )
\end{equation*}
Now, we can apply \eqref{arc_id} and simplify to end with the following identity involving squares of balancing polynomials
\begin{equation}\label{Bal_quad}
\sum_{n=0}^\infty \frac{(-1)^{n}}{2n+1}\frac{B_{2n+1}^2(x)}{z^{2n+1}} 
= \frac{1}{4(9x^2-1)} \arctan\Big (\frac{4(9x^2-1) z(z^2-1)}{z^4+6(12x^2-1)z^2+1} \Big ), \quad z >\lambda_1^2(x).
\end{equation}
Similarly, the same replacement in \eqref{main2C} gives 
\begin{equation}\label{LucBal_quad}
\sum_{n=0}^\infty \frac{(-1)^{n}}{2n+1}\frac{C_{2n+1}^2(x)}{z^{2n+1}} 
= \frac{1}{4}\arctan\Big (\frac{36x^2z(z^2-1)}{z^4-2(36x^2-1)z^2+1}\Big ), \quad z > \lambda_1^2(x).
\end{equation}
To get a series for, say, $\pi/128$ we set $x=1$ in \eqref{Bal_quad} and are left with the quadric
\begin{equation}\label{deg4eqn1}
z^4 - 32z^3 + 66z^2 + 32z + 1 = 0.
\end{equation}
The roots are
\begin{eqnarray*}
	8-\sqrt{47}-4\sqrt{7-\sqrt{47}}, \qquad 8-\sqrt{47}+4\sqrt{7-\sqrt{47}}, \\ 
	8+\sqrt{47}-4\sqrt{7+\sqrt{47}}, \qquad 8+\sqrt{47}+4\sqrt{7+\sqrt{47}}.
\end{eqnarray*}
None of the roots satisfies the condition for convergence and therefore we can conclude that no such series with $B_{2n+1}^2$ exists.
Setting $x=1$ in the above Lucas-balancing identity we obtain 
\begin{equation*}
\frac{\pi}{16} = \sum_{n=0}^\infty \frac{(-1)^{n}}{2n+1}\frac{C_{2n+1}^2}{z^{*2n+1}},
\end{equation*}
where $z^*$ is a positive root of the polynomial equation
\begin{equation}\label{deg4eqn1}
z^4 - 36z^3 - 70z^2 + 36z + 1 = 0.
\end{equation}
Here, also all four roots of equation (\ref{deg4eqn1}) are real
\begin{eqnarray*}
	9-7\sqrt{2}-3\sqrt{20-14\sqrt{2}}, \qquad 9-7\sqrt{2}+3\sqrt{20-14\sqrt{2}}, \\ 
	9+7\sqrt{2}-3\sqrt{20+14\sqrt{2}}, \qquad 9+7\sqrt{2}+3\sqrt{20+14\sqrt{2}}
\end{eqnarray*}
and the only (biggest) root that satisfies the condition $z >\lambda_1^2(1)$ is 
\begin{equation*}
z^* = 9+7\sqrt{2}+3\sqrt{20+14\sqrt{2}} = 37,8254\ldots. 
\end{equation*}
Hence,
\begin{equation}\label{pi33}
\frac{\pi}{16} = \sum_{n=0}^\infty \frac{(-1)^{n}}{2n+1}\frac{C_{2n+1}^2}{(9+7\sqrt{2}+3\sqrt{20+14\sqrt{2}})^{2n+1}}.
\end{equation}

From \eqref{Bal_quad} and \eqref{LucBal_quad}, with $x=L_{2m}/6$, we arrive at
\begin{equation}
\sum_{n=0}^\infty \frac{(-1)^{n}}{2n+1}\frac{F_{2m(2n+1)}^2}{z^{2n+1}} 
= \frac{1}{5} \arctan\Big (\frac{5F_{2m}^2 z(z^2-1)}{z^4+2(5F_{2m}^2+1)z^2+1} \Big )
\end{equation}
and
\begin{equation}
\sum_{n=0}^\infty \frac{(-1)^{n}}{2n+1}\frac{L_{2m(2n+1)}^2}{z^{2n+1}} 
= \arctan\Big (\frac{L_{2m}^2 z(z^2-1)}{z^4-2(L_{2m}^2-1)z^2+1} \Big ).
\end{equation}
To get a new Castellanos' type series for $\pi$ it is necessary to study the equations
\begin{equation}\label{deg4eqn2}
z^4 - L_{2m}^2 z^3 - 2(L_{2m}^2-1)z^2 + L_{2m}^2 z + 1 = 0
\end{equation}
and
\begin{equation}\label{deg4eqn3}
z^4 - 5F_{2m}^2 z^3 + 2(5F_{2m}^2+1)z^2 + 5F_{2m}^2 z + 1 = 0.
\end{equation}

Note that, if in \eqref{deg4eqn2} we set $a=-L_{2m}^2$ then the quartic takes the form
\begin{equation}\label{deg4eqn2_2}
z^4 + a z^3 + 2(a+1)z^2 - a z + 1 = 0.
\end{equation}
The four roots are given by
\begin{equation}\label{quartic_evenLuc1}
z_{1/2} = - \frac{a}{4} - \frac{1}{4} \sqrt{a^2-8a-16} \pm \frac{1}{2}\sqrt{\frac{a^2-4a+a\sqrt{a^2-8a-16}}{2}}
\end{equation}
and
\begin{equation}\label{quartic_evenLuc2}
z_{3/4} = - \frac{a}{4} + \frac{1}{4} \sqrt{a^2-8a-16} \pm \frac{1}{2}\sqrt{\frac{a^2-4a-a\sqrt{a^2-8a-16}}{2}}.
\end{equation}
Finally, if we set $\tilde{t}=-(a+\sqrt{a^2-8a-16})/4$ and $t=-(a-\sqrt{a^2-8a-16})/4$ then
\begin{equation*}
z_{1/2} = \tilde{t} \pm \sqrt{\tilde{t}^2+1} \qquad \mbox{and} \qquad z_{3/4} = t \pm \sqrt{t^2+1}.
\end{equation*}
Among the four roots we choose the biggest root that satisfies the condition for convergence and arrive at the main theorem.

\begin{theorem} \label{thm_quadC}
For each $m\geq 0$ we have the following series for $\pi$ involving squared even-indexed Lucas numbers
\begin{equation}\label{main_quadC}
\frac{\pi}{4} = \sum_{n=0}^\infty \frac{(-1)^{n}}{2n+1}\frac{L_{2m(2n+1)}^2}{(t + \sqrt{t^2+1})^{2n+1}}
\end{equation}
with
\begin{equation}\label{t_value}
t = t(m) = \frac{L_{2m}^2 + \sqrt{L_{2m}^4 + 8L_{2m}^2 - 16}}{4}.
\end{equation}
\end{theorem}

Theorem \ref{thm_quadC} is a wonderful Lucas analogue of Castellanos' equation \eqref{cas2}. 
For $m=0$, $t(0)=1+\sqrt{2}$ and 
\begin{equation}\label{quad_Lucm0}
	\frac{\pi}{16} = \sum_{n=0}^\infty \frac{(-1)^{n}}{2n+1}\frac{1}{(1+\sqrt{2}+\sqrt{4+2\sqrt{2}})^{2n+1}}.
\end{equation}
For $m=1$, $t(1)=(9+\sqrt{137})/4$ and  
\begin{equation}\label{quad_Lucm1}
\frac{\pi}{16} = \sum_{n=0}^\infty \frac{(-1)^{n}}{2n+1}\frac{L_{4n+2}^2 4^{2n}}{(9+\sqrt{137}+3\sqrt{26+2\sqrt{137}})^{2n+1}}.
\end{equation}

The analysis of the quartic \eqref{deg4eqn3} is very similar. With $b=-5F_{2m}^2$ the relevant equation becomes
\begin{equation}\label{deg4eqn3_2}
z^4 + b z^3 + 2(1-b)z^2 - b z + 1 = 0.
\end{equation}
The four roots are given by
\begin{equation*}\label{quartic_evenFib1}
z_{1/2} = - \frac{b}{4} - \frac{1}{4} \sqrt{b^2+8b-16} \pm \frac{1}{2}\sqrt{\frac{b^2+4b+b\sqrt{b^2+8b-16}}{2}}
\end{equation*}
and
\begin{equation*}\label{quartic_evenFib2}
z_{3/4} = - \frac{b}{4} + \frac{1}{4} \sqrt{b^2+8b-16} \pm \frac{1}{2}\sqrt{\frac{b^2+4b-b\sqrt{b^2+8b-16}}{2}}.
\end{equation*}
These roots can be expressed as 
\begin{equation*}
z_{1/2} = \tilde{s} \pm \sqrt{\tilde{s}^2+1} \qquad \mbox{and} \qquad z_{3/4} = s \pm \sqrt{s^2+1}
\end{equation*}
with $\tilde{s}=\tilde{s}(m)=-(b+\sqrt{b^2+8b-16})/4$ and $s=s(m)=-(b-\sqrt{b^2+8b-16})/4$. 
For $m=0$ and $m=1$ we have only complex roots: $\tilde{s}(0)=-i, s(0)=i$ and $\tilde{s}(1)=(5-\sqrt{31}i)/4, s(1)=(5+\sqrt{31}i)/4$
with $i=\sqrt{-1}$. For $m\geq 2$, we get the real roots but no root satisfies the condition $z >\lambda_1^2(L_{2m}/6) = \alpha^{4m}$. 

In a similar manner other arctan arguments can be treated, yielding (if they exist) other series for $\pi$ 
with squares of Fibonacci and Lucas numbers.

\section{Similar series via subseries}

Before closing the investigation, we present some additional results that are closely related to those from Section 2.
The next theorem will be useful to prove some subseries analogues of these results. 

\begin{theorem} \label{thm1_sub}
For each real $x$ with $x>1/3$ and $z\in\mathbb{C}$ with $|z|>|\lambda_1(x)|$, we have the following identity
\begin{eqnarray}\label{main1_sub}
\sum_{n=0}^\infty \frac{1}{4n+3}\frac{B_{4n+3}(x)}{z^{4n+3}} & = &
\frac{1}{8\sqrt{9x^2-1}}\ln\Big (\frac{z^2+2\sqrt{9x^2-1}z-1}{z^2-2\sqrt{9x^2-1}z-1}\Big ) \nonumber \\
&& \quad\quad - \frac{1}{4\sqrt{9x^2-1}}\arctan\Big (\frac{2\sqrt{9x^2-1}z}{z^2+1}\Big )
\end{eqnarray}
and
\begin{equation}\label{main2_sub}
\sum_{n=0}^\infty \frac{1}{4n+3}\frac{C_{4n+3}(x)}{z^{4n+3}} =
\frac{1}{8}\ln\Big (\frac{z^2+6xz+1}{z^2-6xz+1}\Big ) - \frac{1}{4}\arctan\Big (\frac{6xz}{z^2-1}\Big ).
\end{equation}
\end{theorem}
\begin{proof}
From Chen's paper \cite{Chen} we have the relation
\begin{equation}\label{Chen1}
\sum_{n=0}^\infty \frac{z^{4n+3}}{4n+3} = \frac{1}{4}\ln\Big (\frac{1+z}{1-z}\Big ) - \frac{1}{2}\arctan(z), \quad |z|<1.
\end{equation}
The remainder of the proof is the same as in the proof of Theorem \ref{thmC}.
\end{proof}

In the applications we can proceed as in the previous sections. 
First, from \eqref{main1_sub} and \eqref{main2_sub} with $x=1$ it is evident that 
we can obtain some series for $\pi$ with balancing and Lucas-balancing coefficients by 
appropriately choosing the argument of the arctan function on the right. 
The structure of the equations also shows that all series will start with a constant logarithmic term.
These arguments yield to the following exotic representations:
\begin{equation}\label{bm0}
\frac{\pi}{3} = \ln \Big (\frac{46+4\sqrt{138}+8\sqrt{12}+4\sqrt{46}}{46+4\sqrt{138}-8\sqrt{12}-4\sqrt{46}}\Big )
- 16\sqrt{2} \sum_{n=0}^\infty \frac{1}{4n+3} \frac{B_{4n+3}}{(2\sqrt{6}+\sqrt{23} )^{4n+3}},
\end{equation}
\begin{equation}\label{cm0}
\frac{\pi}{3} = \ln \Big (\frac{56+12\sqrt{21}+18\sqrt{3}+12\sqrt{7}}{56+12\sqrt{21}-18\sqrt{3}-12\sqrt{7}}\Big )
- 8 \sum_{n=0}^\infty \frac{1}{4n+3} \frac{C_{4n+3}}{(3\sqrt{3}+2\sqrt{7} )^{4n+3}},
\end{equation}
\begin{eqnarray}\label{bm1}
\frac{\pi}{6} & = & \ln \Big (\frac{2+8\sqrt{3}+4\sqrt{2+8\sqrt{3}}+4(2-\sqrt{3})(4+\sqrt{2+8\sqrt{3}})}
{2+8\sqrt{3}+4\sqrt{2+8\sqrt{3}}-4(2-\sqrt{3})(4+\sqrt{2+8\sqrt{3}})}\Big ) \nonumber\\
&& \qquad - 16\sqrt{2} \sum_{n=0}^\infty \frac{1}{4n+3} B_{4n+3}\Big ( \frac{2-\sqrt{3}}{2\sqrt{2}+\sqrt{1+4\sqrt{3}}} \Big )^{4n+3},
\end{eqnarray}
and
\begin{eqnarray}\label{cm1}
\frac{\pi}{6} & = & \ln \Big (\frac{32-8\sqrt{3}+12\sqrt{4-\sqrt{3}}+6(2-\sqrt{3})(3+2\sqrt{4-\sqrt{3}})}
{32-8\sqrt{3}+12\sqrt{4-\sqrt{3}}-6(2-\sqrt{3})(3+2\sqrt{4-\sqrt{3}})}\Big ) \nonumber\\
&& \qquad - 8 \sum_{n=0}^\infty \frac{1}{4n+3} C_{4n+3}\Big ( \frac{2-\sqrt{3}}{3+2\sqrt{4-\sqrt{3}}} \Big )^{4n+3}.
\end{eqnarray}

Working with Fibonacci and Lucas numbers we can derive the next series, 
where each series is stated as a separate theorem without proof.

\begin{theorem} \label{thm2_sub}
For each integer $m\geq 0$ the following series identity holds
\begin{eqnarray}\label{main3_sub}
\frac{\pi}{2} & = & \ln\Big (\frac{1}{2}\Big (5F_{2m+1}^2 + \sqrt{5}F_{2m+1}\sqrt{5F_{2m+1}^2+4} + 2 \Big )\Big ) \nonumber \\
&& \qquad - 4\cdot \sqrt{5} \sum_{n=0}^\infty \frac{1}{4n+3} F_{(2m+1)(4n+3)} \Big (\frac{2}{\sqrt{5}F_{2m+1}
+ \sqrt{5F_{2m+1}^2+4}}\Big )^{4n+3}.
\end{eqnarray}
\end{theorem}

The series in Theorem \ref{thm2_sub} is the Fibonacci subseries analogue of Castellanos' series \eqref{cas}.

When $m=0$ the expression reads as
\begin{equation*}
\frac{\pi}{2} = \ln\Big (\frac{7+3\sqrt{5}}{2} \Big ) 
- 4\sqrt{5} \sum_{n=0}^\infty \frac{1}{4n+3} F_{4n+3} \Big (\frac{2}{3 + \sqrt{5}}\Big )^{4n+3},
\end{equation*}
which has the equivalent form
\begin{equation}
\frac{\pi}{8} = \ln (\alpha) - \sqrt{5} \sum_{n=0}^\infty \frac{1}{4n+3} \frac{F_{4n+3}}{\alpha^{8n+6}}.
\end{equation}

We conclude with three other examples exhibiting this structure and leave the derivations to the interested reader.

\begin{theorem} \label{thm3_sub}
For each integer $m\geq 0$ the following expression for $\pi$ is valid
\begin{equation}\label{main4_sub}
\frac{\pi}{2} = \ln\Big (\frac{1}{2}\Big (L_{2m}^2 + L_{2m}\sqrt{L_{2m}^2+4} + 2 \Big )\Big )
- 4 \sum_{n=0}^\infty \frac{1}{4n+3} L_{2m(4n+3)} \Big (\frac{2}{L_{2m} + \sqrt{L_{2m}^2+4}}\Big )^{4n+3}.
\end{equation}
\end{theorem}

\begin{theorem} \label{thm4_sub}
For each integer $m\geq 0$, the following identity for $\pi$ is valid
\begin{eqnarray}\label{main5_sub}
\frac{\pi}{6} &=& \ln\left(\frac{(9+5\sqrt{3})L_{2m}^2+(3+\sqrt{3})L_{2m} \sqrt{(2+\sqrt{3})^2L_{2m}^2+4} +4} 
{(5+3\sqrt{3})L_{2m}^2+(1+\sqrt{3})L_{2m}\sqrt{(2+\sqrt{3})^2L_{2m}^2+4} +4}\right) \nonumber \\
& & - 4\sum_{n=0}^\infty \frac{1}{4n+3} L_{2m(4n+3)} \Big (\frac{2}{(2+\sqrt{3})L_{2m} + \sqrt{(2+\sqrt{3})^2L_{2m}^2+4}}\Big )^{4n+3}.
\end{eqnarray}
\end{theorem}

\begin{theorem} \label{thm5_sub}
For each integer $m\geq 0$, we have the following expression for $\pi$
\begin{eqnarray}\label{main6_sub}
\frac{\pi}{6} &=& \ln\left(\frac{(9+5\sqrt{3})\sqrt{5}\Big(\sqrt{5}F_{2m+1}^2+F_{2m+1} \sqrt{5F_{2m+1}^2+4(2-\sqrt{3})^2}\Big) +4} 
{(5+3\sqrt{3})\sqrt{5}\Big(\sqrt{5}F_{2m+1}^2+F_{2m+1} \sqrt{5F_{2m+1}^2+4(2-\sqrt{3})^2}\Big) +4}\right) \nonumber \\
& & - 4\sqrt{5}\sum_{n=0}^\infty \frac{1}{4n+3} F_{(2m+1)(4n+3)}\Big(\frac{2(2-\sqrt{3})}{\sqrt{5}F_{2m+1} + \sqrt{5F_{2m+1}^2+4(2-\sqrt{3})^2}}\Big )^{4n+3}.
\end{eqnarray}
\end{theorem} 	
 
\section{Conclusion}

Castellanos has derived curious series for $\pi$ involving odd Fibonacci and odd squared Fibonacci numbers.
The goal of the paper was to provide complements of his work. This was achieved by studying some infinite series involving
balancing and Lucas-balancing polynomials. We have also indicated how more series representations of this nature can be  
established.


\section{Appendix: Numerical experiments}

All series that have been derived in this paper have been verified numerically for different values of 
the parameters $m$ and $n$ using MATLAB. Here, we present some of the numerical results we obtained, i.e., 
our results concerning the behavior of the partial sums, paying particular attention to Sections 2 and 3. 
The reference value for $\pi$ is
\begin{equation*}
3.14159265358979323846264338327950288419716939937510582097494459230781...
\end{equation*}
We begin with a comparison of the series \eqref{cas2_1} and \eqref{Luc_bal_pi}, 
which is summarized in the first table. 

\begin{table}[h!]
	\begin{tabular}{c|c| c }
		\hline
		\textbf{\small n} &\textbf{\small Series \eqref{cas2_1} (correct digits of $\pi$)} & \textbf{\small Series \eqref{Luc_bal_pi} (correct digits for $\pi$)} \\
		\hline
		\small 10 & \small 3.1415926540096031579541983575759 (8) & \small 3.1681253658683134338813758290598 (1) \\
		\small 25 & \small 3.1415926535897931755199500028705 (15) & \small 3.1394749981715568271090077701047 (1) \\
		\small 50 & \small 3.1415926535897931755212222064311 (15) & \small 3.1416593510252682593586123771523 (3) \\
		\small 75 & \small 3.1415926535897931755212222064311 (15) & \small 3.1415898900113144709178285437189 (4) \\
		\small 100 & \small 3.1415926535897931755212222064311 (15) & \small 3.1415927819755543664446705866538 (6) \\ \hline
	\end{tabular}
	\centering 
	\caption{Comparison of series \eqref{cas2_1} and series \eqref{Luc_bal_pi}}
\end{table}

Series \eqref{cas2_1} converges much faster to $\pi$ than \eqref{Luc_bal_pi}. For instance, for $n=25$
the partial sum equals 3.1415926535897931755199500028705 and gives 15 correct decimal places of $\pi$, 
whereas series (\ref{Luc_bal_pi}) only gives one correct decimal place in this case. 
Even for the big input $n=100$ \eqref{Luc_bal_pi} is able to get only 6 correct decimal places for $\pi$. 

The series \eqref{main2LucC}-\eqref{main5LucC} and their Fibonacci counterparts from \eqref{main1Fib}-\eqref{main3Fib} 
exhibit much better convergence properties as is seen from the next tables. 
Finally, we state the numerical results for the series containing squared arguments in the numerator.

\begin{table}[h!]
	\resizebox{\textwidth}{!}{
\begin{tabular}{c|c|c|c}
\hline
$n$ & {Series} & Partial sum and correct decimal places & Partial sum and correct decimal places \\
			 &	 & of $\pi$, when $m=0$    & of $\pi$, when $m=1$ \\
\hline
$n=10$
& \eqref{main2LucC} & 3.1415926540617729889571081190938 (8) &  3.1421196570445914842979046044391 (2) \\
& \eqref{main3LucC} & 3.14159265358982738004655360471 (12) & 3.1415926671660013427841119922951 (7) \\
& \eqref{main4LucC} & 3.1415926535897947339924377936013 (14) & 3.1415926535897956856265976351507 (14) \\
& \eqref{main5LucC} & 3.1415926535921310968341963746407 (10) & 3.1415939762233089574900159429531 (5) \\
\hline
$n=25$
& \eqref{main2LucC} & 3.1415926535897935109299308733563 (15) & 3.1415924425521257791274315531176 (6)\\
& \eqref{main3LucC} & 3.141592653589793111803909553252 (15) & 3.1415926535897930747899824982812 (15)\\
& \eqref{main4LucC} & 3.1415926535897947339867034604622 (14) & 3.1415926535897944192246093796406 (14)\\
& \eqref{main5LucC} & 3.141592653589792871679833893134 (14) & 3.1415926535896632627309996371646 (12)\\
\hline
$n=50$ 
& \eqref{main2LucC} & 3.1415926535897935109305996238453 (15) & 3.1415926535907651359119064125043 (10)\\
& \eqref{main3LucC} & 3.1415926535897931118039095532521 (15) & 3.1415926535897930771752670402775 (15)\\
& \eqref{main4LucC} &  3.1415926535897947339867034604622 (14) & 3.1415926535897944192246093796406 (14)\\
& \eqref{main5LucC} & 3.1415926535897928716798338954132 (14) & 3.1415926535897928493289793465367 (14)\\
\hline
$n=75$ 
& \eqref{main2LucC} & 3.1415926535897935109305996238453 (15) & 3.1415926535897932445405743444239 (16)\\
& \eqref{main3LucC} & 3.1415926535897931118039095532521 (15) & 3.1415926535897930771752670402775 (15)\\
& \eqref{main4LucC} &  3.1415926535897947339867034604622 (14) & 3.1415926535897944192246093796406 (14)\\
& \eqref{main5LucC} & 3.1415926535897928716798338954132 (14) & 3.1415926535897928493289787796672 (14)\\
\hline
$n=100$
& \eqref{main2LucC} & 3.1415926535897935109305996238453 (15) & 3.1415926535897932504225822917772 (16)\\
& \eqref{main3LucC} & 3.1415926535897931118039095532521 (15) & 3.1415926535897930771752670402775 (15)\\
& \eqref{main4LucC} &  3.1415926535897947339867034604622 (14) & 3.1415926535897944192246093796406 (14)\\
& \eqref{main5LucC} & 3.1415926535897928716798338954132 (14) & 3.1415926535897928493289787796672 (14) \\
\hline
\end{tabular}}
\centering
\caption{Series with even indexed Lucas coefficients}
\end{table}

\begin{table}[h!]
	\resizebox{\textwidth}{!}{
		\begin{tabular}{c|c|c|c}
			\hline
			$n$ & {Series} & Partial sum and correct decimal places & Partial sum and correct decimal places \\
			&	 & of $\pi$, when $m=0$    & of $\pi$, when $m=1$ \\
			\hline
			$n=10$
			& \eqref{main1Fib} & 3.1415926535897945091605064246388 (14) & 3.1415926535898033991093222132522 (12) \\
			& \eqref{main2Fib} & 3.1415926536976031512950065859936 (9) & 3.1415927863118601169382604268352 (6) \\
			& \eqref{main3Fib} & 3.1415926617777441429048832649656 (7) & 3.1416085221242421932891054056989 (3) \\
			& \eqref{cas} & 3.1415946628008058070270975343262 (5) & 3.1513648928833562151579600460063 (1) \\
			\hline
			$n=25$
			& \eqref{main1Fib} & 3.1415926535897944941490876249881 (14) & 3.1415926535897942690609285922555 (14) \\
			& \eqref{main2Fib} & 3.1415926535897929748407090095711 (14) & 3.1415926535897928788785708458934 (14) \\
			& \eqref{main3Fib} & 3.1415926535897927600555556635439 (14) & 3.1415926535469317102875642213194 (10) \\
			& \eqref{cas} & 3.1415926535893304765153112605509 (12) & 3.1413911371820930200840194675568 (3) \\
			\hline
			$n=50$
			& \eqref{main1Fib} & 3.1415926535897944941490876249881 (14) & 3.1415926535897942690609285922555 (14) \\
			& \eqref{main2Fib} & 3.141592653589792974840740857073 (14) & 3.1415926535897933582722721433178 (15) \\
			& \eqref{main3Fib} & 3.1415926535897927609931468741211 (14) & 3.1415926535897929816708906429346 (14) \\
			& \eqref{cas} & 3.1415926535897932529555146219012 (16) & 3.1415933178255789377527632793825 (5) \\
			\hline
			$n=75$ 
			& \eqref{main1Fib} & 3.1415926535897944941490876249881 (14) & 3.1415926535897942690609285922555 (14) \\
			& \eqref{main2Fib} & 3.141592653589792974840740857073 (14) & 3.1415926535897933582722721433093 (15) \\
			& \eqref{main3Fib} & 3.1415926535897927609931468741211 (14) & 3.1415926535897929816235883662303 (14) \\
			& \eqref{cas} & 3.1415926535897932529555062020827 (16) & 3.1415926507102507935365145416425 (8) \\
			\hline
			$n=100$
			& \eqref{main1Fib} & 3.1415926535897944941490876249881 (14) & 3.1415926535897942690609285922555 (14) \\
			& \eqref{main2Fib} & 3.141592653589792974840740857073 (14) & 3.1415926535897933582722721433093 (15) \\
			& \eqref{main3Fib} & 3.1415926535897927609931468741211 (14) & 3.1415926535897929816235883662988 (14) \\
			& \eqref{cas} & 3.1415926535897932529555062020827 (16) & 3.1415926536037885968968215682418 (9) \\
			\hline
	\end{tabular}}
	\centering
	\caption{Series with odd indexed Fibonacci coefficients including Castellaons series \eqref{cas}}
\end{table} 

\begin{table}[h!]
	\begin{tabular}{c|c }
		\hline
		\textbf{\small n} &\textbf{\small Series \eqref{pi33} (correct digits of $\pi$) }\\ \hline
		\small 10 & \small 3.1500288764501782022564579932578 (1) \\
		\small 25 & \small 3.1414500748004593281290444672926 (3) \\
		\small 50 & \small 3.1415929908436770508682631082098 (6) \\
		\small 75 & \small 3.1415926525406591485619208477698 (8) \\
		\small 100 & \small 3.1415926535934520603658154364136 (10) \\ \hline
	\end{tabular}
	\centering 
	\caption{Series with odd squared Lucas-balancing coefficients}
\end{table}

\begin{table}[h]
	\begin{tabular}{c|c| c }
		\hline
		\textbf{\small n} &\textbf{\small Series \eqref{main_quadC} when $m=0$,} & \textbf{\small Series \eqref{main_quadC} when $m=1$} \\
		&	\small (correct digits of $\pi$) & \small (correct digits of $\pi$) \\
		\hline
		\small 10 & \small 3.141592653589793277636375358286 (16) & \small 3.1416003226810113749580541993994 (3) \\
		\small 25 & \small 3.1415926535897932279669723561782 (16) & \small 3.1415926535792334445640011923928 (10) \\
		\small 50 & \small 3.1415926535897932279669723561782 (16) & \small 3.1415926535897932072655992250783 (16) \\
		\small 75 & \small 3.1415926535897932279669723561782 (16) & \small 3.1415926535897932072618102929481 (16) \\
		\small 100 & \small 3.1415926535897932279669723561782 (16) & \small 3.1415926535897932072618102929499 (16) \\ \hline
	\end{tabular}
	\centering 
	\caption{Comparison of series \eqref{quad_Lucm0} and series \eqref{quad_Lucm1}}
\end{table}




\end{document}